\documentclass{amsart}
\usepackage{amssymb,color}
\usepackage{ogonek}
\RequirePackage[colorlinks,citecolor=blue,urlcolor=blue,pagebackref=false]{hyperref}

\numberwithin{equation}{section}

\newtheorem{theorem}{Theorem}[section]
\newtheorem{lemma}[theorem]{Lemma}

\theoremstyle{definition}

\theoremstyle{remark}
\newtheorem{remark}[theorem]{Remark}

\newcommand{\E}{\mathbb{E}}

\newcommand{\R}{\mathbb{R}}
\def\dx{\operatorname{d}}

\begin{document}

	\title[Non-uniform Berry--Esseen bounds for exchangeable pairs]{Non-uniform Berry--Esseen bounds for exchangeable pairs with applications to the mean-field classical $N$-vector models and Jack measures}

\author[L. V. Th\`{a}nh]{L\^{e} V\v{a}n Th\`{a}nh}
\address{Department of Mathematics, Vinh University, 182 Le Duan, Vinh, Nghe An, Vietnam}
\email{levt@vinhuni.edu.vn}

\author{Nguyen Ngoc Tu}
\address{Faculty of Applied Sciences,  HCMC University of Technology and Education, 01 Vo Van Ngan, Ho Chi Minh City, Viet Nam}
\email{tunn@hcmute.edu.vn}

\subjclass[2020]{Primary 60F05}

\keywords{Stein's method, Exchangeable pair, Kolmogorov distance, Non-uniform bound, Mean-field  classical $N$-vector model, Jack measure}

\date{}

\begin{abstract}
	This paper establishes a non-uniform Berry--Esseen bound in normal approximation for exchangeable pairs using Stein's method via a 
	concentration inequality approach. The main theorem extends and improves several results in the literature, including those of
	Eichelsbacher and L\"{o}we [\href{https://projecteuclid.org/journals/electronic-journal-of-probability/volume-15/issue-none/Steins-Method-for-Dependent-Random-Variables-Occuring-in-Statistical-Mechanics/10.1214/EJP.v15-777.full}{Electron. J. Probab. 15, 2010, 962--988}], and Eichelsbacher [\href{https://arxiv.org/abs/2404.07587}{arXiv:2404.07587, 2024}]. 
		The result is applied to obtain a non-uniform Berry--Esseen bound for the squared-length of the total spin in the mean-field classical $N$-vector models, and a
	non-uniform Berry--Esseen bound for Jack deformations of the character ratio. 
\end{abstract}

\maketitle

		\section{Introduction and the main result}\label{sec.intro}

Stein's method was introduced by Charles Stein \cite{stein1986approximate} to
estimate the distance between a statistic of interest and the
normal distribution. This method is of particular interest because it works well for dependent random variables 
and has applications in various areas of mathematics such as
combinatorics \cite{fulman2004stein,fulman2006inductive}, number theory \cite{chen2023generalized,harper2009two},
and random matrix and random graph theory \cite{lambert2019quantitative, pekoz2013degree}, among others. 
We refer to Stein \cite{stein1986approximate} and Chen et al. \cite{chen2011normal}
for comprehensive expositions on Stein's method for normal approximation.

Recall that
$(W, W')$ is an \textit{exchangeable pair} if two random vectors $(W,W')$ and $(W',W)$ have the same distribution. 
Let $(W, W')$ be an exchangeable pair. In \cite[Theorem 1, Lecture III]{stein1986approximate}, Stein proved a normal approximation for $W$ under the linear regression condition
\begin{equation}\label{linear}
	\E(W'|W) = (1-\tau)W
\end{equation}
for some $0<\tau<1$. The linear regression condition \eqref{linear} can be relaxed to a more flexible condition that
\begin{equation}\label{nonlinear}
	\E(W'|W) = (1-\tau)W+R
\end{equation}
for some random variable $R=R(W)$ and for some $0<\tau<1$. 
Eichelsbacher and L\"{o}we \cite{eichelsbacher2010stein} proved that if $(W,W')$ is an exchangeable pair satisfying \eqref{nonlinear},
then for any $a>0$
\begin{equation}\label{el_main}  
	\begin{aligned}  
		\sup_{z\in\mathbb{R}}|\mathbb{P}(W\le z) - \Phi(z)| &\le \sqrt{\E\left( 1- \dfrac{1}{2\tau}\E(\Delta^2|W)\right)^2}+\left(\dfrac{\sqrt{2\pi}}{4}+1.5a\right)\dfrac{\sqrt{\E R^2}}{\tau}\\
		&\quad+
		1.5a + \dfrac{0.41a^3}{\tau}+
		\dfrac{1}{2\tau}\E\left(\Delta^2 \mathbf{1}(|\Delta|>a)\right),
	\end{aligned}
\end{equation}
where $\Delta=W-W'$. Here and hereafter, $\Phi(z)=\left(1/\sqrt{2\pi}\right)\int_{-\infty}^z e^{-t^2/2}\dx t$ is the distribution function of the standard normal distribution.

Besides dealing with uniform bounds, some authors also used Stein's method to provide non-uniform bounds or moderate deviations.
We refer to Chen and Shao \cite{chen2001non} for non-uniform bounds in normal approximation with
independent summands, and to Chen et al. \cite{chen2013stein}, Fang et al. \cite{fang2020refined}, and
Liu and Zhang \cite{liu2023cramer}
for the Cram\'{e}r-type moderate deviations.
Chen et al. \cite{chen2021error} and Eichelsbacher \cite{eichelsbacher2024stein} used Stein's method to obtain non-uniform Berry--Esseen bounds
for zero-biased couplings and exchangeable pairs, respectively. Very recently, in Butzek and Eichelsbacher \cite{butzek2024non} and Dung et al. \cite{dung2024non},
the authors
obtained non-uniform Berry--Esseen bounds for normal approximations by the Malliavin--Stein method. The following result is Theorem 2.3 of Eichelsbacher \cite{eichelsbacher2024stein}. 

\begin{theorem}[Eichelsbacher  \cite{eichelsbacher2024stein}]\label{thm.eich}
	Let  $(W,W')$ be an exchangeable pair satisfying \eqref{nonlinear}, and let $\Delta=W-W'$. Assume that $\E W^{2}\le 2$ and $|\Delta|\le a$ for some $a>0$. Then for any $z\in\mathbb{R}$, we have
	\begin{equation}\label{eiche03}  
		\begin{split}  
			|\mathbb{P}(W\le z) - \Phi(z)|&\le \dfrac{C}{1+|z|}\left(\E\left|1- \dfrac{1}{2\tau} \E(\Delta^2|W)\right|+\dfrac{\E |R|}{\tau}+3a\right),
		\end{split}
	\end{equation} 
	where $C$ is a constant depending only on $\E W^2$.
\end{theorem}

The proof of Theorem \ref{thm.eich} is quite simple and 
the term $1/(1+|z|)$ in \eqref{eiche03} can be improved to $1/(1+|z|^p)$ 
under the condition that $\E|W|^{2p}<\infty$ for $p \ge 1$. 
However, if the remainder $R$ appears as in \eqref{nonlinear} and $\Delta$ is not bounded, the problem
becomes more challenging (see, Eichelsbacher \cite[Remark 2]{eichelsbacher2024stein}).
In many applications, the exchangeable pair $(W,W')$
only satisfies \eqref{nonlinear} rather than \eqref{linear}, and $|\Delta|$ may not be bounded. 
In this paper, we obtain a sharp non-uniform
Berry--Esseen bound for unbounded exchangeable pairs under condition \eqref{nonlinear}. 
Our approach uses a non-uniform concentration inequality, which differs from the aforementioned papers on non-uniform Berry--Esseen bounds.
Moreover, in our result, the dependence on $\E|W|^{2p}$ is explicit.
The main result of the paper is the following theorem.

\begin{theorem} \label{thm.main11}
	Let $p\ge 1$ and  $a>0$. Let $(W,W')$ be an exchangeable pair satisfying \eqref{nonlinear}, and let $\Delta=W-W'$.
	Then for every $z\in\mathbb{R}$, we have
	\begin{equation*}\label{nonuni.main}
		\begin{split}  
			&|\mathbb{P}(W\le z) - \Phi(z)|\\
			&\quad\le \dfrac{C(p)\left(1+\E|W|^{2p}\right)}{(1+|z|)^p}\left(\sqrt{\E\left(1- \dfrac{1}{2\tau} \E(\Delta^2|W)\right)^2 } +\dfrac{\sqrt{\E R^2}}{\tau}+a+\dfrac{a^3}{\tau}+\dfrac{\sqrt{\E\Delta^4\mathbf{1}(|\Delta|>a)}}{\tau}\right),
		\end{split}
	\end{equation*}
	where $C(p)$ is a constant depending only on $p$.
\end{theorem} 
We apply Theorem \ref{thm.main11} to obtain a non-uniform Berry--Esseen bound in the central limit theorem
for the squared-length of the total spin of the mean-field classical $N$-vector models and a non-uniform Berry--Esseen bound in the central limit theorem
for Jack deformations of the character ratio. 
The former improves the main result of Th\`{a}nh and Tu \cite{thanh2019error}, 
while the latter recovers a result of Chen et al. \cite{chen2021error} by a different method.

It is worth noting that we do not assume $|W-W'|\le a$ as in Theorem \ref{thm.eich},  thereby extending the applicability of our result to unbounded exchangeable pairs
(see Subsection \ref{subsec.jack} for the case of the Jack deformations of the character ratio).

Throughout this paper, $C$ denotes a universal constant whose value may change from line to line. 
The symbol $C(\cdot)$ denotes a constant that depends only on the variables inside the parentheses, and its value may also change from line to line.
For $x\in\R$, the natural logarithm (base $e=2.7182\ldots$) of $\max\{x,e\}$ will be denoted by $\log x$.
The supremum norm of a function $f$ is denoted by $\|f\|$.

The rest of the paper is organized as follows.
In Section \ref{sec.proof}, we provide four preliminary lemmas and the proof of Theorem \ref{thm.main11}.
Non-uniform Berry--Esseen bounds for the squared-length of the total spin in the mean-field classical $N$-vector models and Jack measures are presented in Section \ref{sec.appl}.
Finally, Appendix \ref{sec.proof.lem} includes a remark on the Jack measures and the proofs of four preliminary lemmas.

\section{Proof of main result}\label{sec.proof}

In this section, we will prove the main result of the paper. 
Throughout, we use all notations as defined in Theorem \ref{thm.main11}.
If $\E|W|^{2p}=\infty$, then the result is trivial. Therefore,
it suffices to consider the case $\E|W|^{2p}<\infty$.
For $z\in\R$, let $f_z$ be the solution to the Stein equation 
\begin{equation}\label{s3} 
	f'(x) - xf(x) = \mathbf{1}(x \le z) - \Phi(z). 
\end{equation}

We will need the following four lemmas. 
The proofs of these lemmas
are deferred to the Section \ref{sec.proof.lem}.
In Lemmas \ref{lem.bound.moment.W} and \ref{lem.bsg}, the number $100$ does not play any special role.	

\begin{lemma}\label{lem.bound.moment.W} 
	Let $0<q\le 2p$ and let $|c|\le 100$. Then
	\begin{equation}\label{estimate.moment.W01}
		\E|W+c|^q\le C(p)(1+\E|W|^{2p}).
	\end{equation}
\end{lemma} 

The following lemma is a non-uniform concentration inequality for $W$, which plays an important role in the proof of Theorem \ref{thm.main11}.
For uniform concentration inequalities for exchangeable pairs, we refer to Chen and Fang \cite{chen2015error}, Eichelsbacher and L\"{o}we \cite{eichelsbacher2010stein}, 
and Shao and Su \cite{shao2006berry}.

\begin{lemma}\label{lem.concentration.inequality} 
	Let $z \ge 0$. Then for $0<a< 1$, we have
	\begin{align} \label{T322a}
		\E \left(\Delta^2\mathbf{1}(|\Delta|\le a)\mathbf{1}(z-a \le W \le z+a)\right) \le \dfrac{C(p)(1+\E|W|^{2p})a(\sqrt{\E R^2} + \tau)}{(1+z)^p}.
	\end{align}
\end{lemma} 

Lemmas \ref{lem.bound.solutions1} and \ref{lem.bsg} are similar to Lemmas 5.1 and 5.2 of Chen and Shao \cite{chen2001non}, respectively.

\begin{lemma}\label{lem.bound.solutions1} 
	Let $z \ge 5$ and let $\xi$ be a random variable satisfying $0\le |\xi|\le |\Delta|$. Then
	\begin{equation}\label{estimate.lem1.01}
		\E(f_{z}(W))^2 \le \dfrac{C(p)(1+\E|W|^{2p})}{(1+z)^{2p}}
	\end{equation}
	and
	\begin{equation}\label{estimate.lem1.02}
		\E(f_{z}'(W+\xi))^2 \le \dfrac{C(p)(1+\E|W|^{2p})}{(1+z)^{2p}}.
	\end{equation}
\end{lemma} 

\begin{lemma} \label{lem.bsg}  
	Let $z\ge 5$, $|u| \le 100$ and $g_{z}(w) = (wf_z(w))'=f_{z}(w)+wf_{z}'(w)$. Then
	$$\E g_{z}(W+u) \le \dfrac{C(p)(1+\E|W|^{2p})}{(1+z)^p}.$$
\end{lemma} 

We are now ready to prove the main theorem.

\begin{proof}[Proof of Theorem \ref{thm.main11}]
	To bound $|\mathbb{P}(W\le z) - \Phi(z)|$, it suffices to consider $z \ge 0$ since we can simply apply the result to $-W$ when $z<0$. In view of \eqref{el_main} and the fact that
	\[\E(\Delta^2\mathbf{1}(|\Delta|>a))\le \sqrt{\E(\Delta^4\mathbf{1}(|\Delta|>a))},\] 
	we may, without loss of generality, assume that $z\ge 5$. Since
	\begin{equation*}
		\begin{split}
			|\mathbb{P}(W\le z) - \Phi(z)|&=|\mathbb{P}(W> z) - (1-\Phi(z))|\\
			&\le \dfrac{\E|1+W|^{2p}}{(1+z)^{2p}}+(1-\Phi(z))\\
			&\le \dfrac{C(p)\left(1+\E|W|^{2p}\right)}{(1+|z|)^p},
		\end{split}
	\end{equation*}
	the conclusion of the theorem holds if $a\ge 1$. Thus, it remains to consider the case $0<a<1$.
	For any absolutely continuous function $f$, we obtain from the exchangeability and \eqref{nonlinear} that
	\begin{equation}\label{s}
		\begin{aligned}
			\E\left((W-W')(f(W)-f(W'))\right)&= 2\E Wf(W)-2\E W'f(W)\\
			&= 2\E Wf(W)-2\E f(W)((1-\tau)W+R)\\
			&=2\tau \E Wf(W)-2\E f(W)R.
		\end{aligned}
	\end{equation}
	We recall that $f_{z}$ is the solution to the Stein equation \eqref{s3}. 
	Since \eqref{s} holds for $f_{z}$, we have
	\begin{equation}\label{s4}
		\begin{aligned}
			\mathbb{P}(W \le z) - \Phi(z)
			&= \E(f_{z}'(W) - Wf_{z}(W))   \\
			&= \E f_{z}'(W) - \dfrac{1}{2\tau} \E(W-W')(f_{z}(W)-f_{z}(W')) - \dfrac{1}{\tau} \E(f_{z}(W)R)   \\
			&= \E\left(f_{z}'(W)(1- \dfrac{1}{2\tau}(W - W')^2)\right) - \dfrac{1}{\tau} \E(f_{z}(W)R) \\
			& \quad  - \dfrac{1}{2\tau} \E\left((W-W')(f_{z}(W)-f_{z}(W')- (W-W')f_{z}'(W))\right) \\    
			&:= T_1 + T_2 +T_3.					  	  
		\end{aligned} 
	\end{equation} 
	It follows from the Cauchy--Schwarz inequality and Lemma \ref{lem.bound.solutions1} that
	\begin{align} \label{b01}
		\begin{aligned}
			|T_1| &\le \sqrt{\E(f_{z}'(W))^2} \sqrt{\E\left( 1- \dfrac{1}{2\tau}\E((W-W')^2|W)\right)^2 } \\
			&\le \dfrac{C(p)(1+\E|W|^{2p})}{(1+z)^p }\sqrt{\E\left( 1- \dfrac{1}{2\tau}\E((W-W')^2|W)\right)^2 },
	\end{aligned} \end{align}
	and
	\begin{align} \label{b02}
		\begin{aligned}
			|T_2| &\le \dfrac{1}{\tau} \sqrt{\E(f_{z}(W))^2} \sqrt{\E(R^2)}\\
			&\le \dfrac{C(p)(1+\E|W|^{2p})}{(1+z)^p}\dfrac{\sqrt{\E(R^2)}}{\tau}.
		\end{aligned}
	\end{align}
	It remains to bound $T_3$. By recalling $\Delta=W-W'$, we have
	\begin{align} \label{s5}
		\begin{aligned}
			(-2\tau)T_3 
			&= \E\left(\Delta(f_{z}(W)-f_{z}(W-\Delta)- \Delta f_{z}'(W))\right)   \\
			&:= T_{3,1} + T_{3,2},	
		\end{aligned}
	\end{align}
	where
	\[T_{3,1}=\E\left(\Delta \mathbf{1}(|\Delta|>a)(f_{z}(W)-f_{z}(W-\Delta)- \Delta f_{z}'(W))\right)\]
	and
	\[T_{3,2}=\E\left(\Delta \mathbf{1}(|\Delta|\le a)(f_{z}(W)-f_{z}(W-\Delta)- \Delta f_{z}'(W))\right).\]
	By applying Taylor's expansion and the Cauchy--Schwarz inequality, we have
	\begin{equation}\label{b030}
		\begin{aligned}
			|T_{3,1}| &= |\E\left(\Delta^2 \mathbf{1}(|\Delta|>a)(f_{z}'(W+\xi)- f_{z}'(W))\right)| \\
			&\le |\E\left(\Delta^2 \mathbf{1}(|\Delta|>a)(f_{z}'(W+\xi)\right)|\\
			&\qquad +|\E\left(\Delta^2 \mathbf{1}(|\Delta|>a) f_{z}'(W))\right)|  \\	
			&\le \left(\E\Delta^4 \mathbf{1}(|\Delta|>a)\E(f_{z}'(W+\xi))^2\right)^{1/2}\\
			&\qquad +\left(\E\Delta^4 \mathbf{1}(|\Delta|>a)\E(f_{z}'(W))^2\right)^{1/2},
		\end{aligned} 
	\end{equation}
	where $\xi$ is a random variable satisfying $0\le |\xi| \le |\Delta|$.
	Combining \eqref{estimate.lem1.02} and \eqref{b030} yields
	\begin{equation}\label{b031}
		\begin{aligned}
			|T_{3,1}|&\le \dfrac{C(p)(1+\E|W|^{2p})\sqrt{\E\Delta^4 \mathbf{1}(|\Delta|>a)}}{(1+z)^p}. 
		\end{aligned} 
	\end{equation}
	From the Stein equation \eqref{s3}, we have
	\begin{equation}\label{b032}
		\begin{aligned}
			T_{3,2} &=  \E\left(\Delta \mathbf{1}(|\Delta| \le a)\int_{-\Delta}^{0}(f_{z}'(W+t) - f_{z}'(W))\dx t \right)   \\
			&=  \E\left( \Delta \mathbf{1}(|\Delta| \le a)\int_{-\Delta}^{0}(W+t)f_{z}(W+t) - Wf_{z}(W)\dx t \right)   \\
			& \quad +\E\left( \Delta \mathbf{1}(|\Delta| \le a)\int_{-\Delta}^{0}(\mathbf{1}(W+t \le z) - \mathbf{1}(W \le z))\dx t\right)  \\
			&:= T_{3,2,1} + T_{3,2,2}.
		\end{aligned} 
	\end{equation}
	By using Lemma \ref{lem.bsg} and noting that $g_z(w)\ge 0$ for all $w$, we have
	\begin{equation} \label{b0321}
		\begin{aligned}
			|T_{3,2,1}| &= \left|\E\left( \Delta \mathbf{1}(|\Delta| \le a)\int_{-\Delta}^{0} \int_{0}^{t}g_{z}(W+u)\dx u \dx t \right)\right|  \\
			&\le  \left|\E\left( \Delta \mathbf{1}(0 \le \Delta \le a)\int_{-\Delta}^{0}\int_{0}^{t}g_{z}(W+u)\dx u \dx t \right)\right|\\ 
			&\quad+ \left|\E\left( \Delta \mathbf{1}(-a \le \Delta \le 0)\int_{-\Delta}^{0}\int_{0}^{t}g_{z}(W+u)\dx u \dx t \right)\right|\\
			&\le  \left| \E\left( \Delta \mathbf{1}(0 \le \Delta \le a)\int_{0}^{\Delta}\int_{-t}^{0}g_{z}(W+u)\dx u \dx t \right) \right| \\
			&\quad+ \left|\E\left( \Delta \mathbf{1}(-a \le \Delta \le 0)\int_{0}^{-\Delta}\int_{0}^{t}g_{z}(W+u)\dx u \dx t \right)\right|\\ 
			&\le  a\int_{0}^{a}\int_{-t}^{0}\E g_{z}(W+u)\dx u \dx t  + a\int_{0}^{a}\int_{0}^{t}\E g_{z}(W+u)\dx u \dx t \\
			&= 	a \int_{0}^{a}\int_{-t}^{t}\E g_{z}(W+u)\dx u \dx t\\
			& \le  \dfrac{C(p)(1+\E|W|^{2p})a^3}{(1+z)^p}.
		\end{aligned} 
	\end{equation}		
	Similarly,
	\begin{equation} \label{b0324}
		\begin{aligned}
			|T_{3,2,2}|&\le\left|\E\left( \Delta \mathbf{1}(0\le \Delta\le a)\int_{-\Delta}^{0}\mathbf{1}(z< W \le z-t)\dx t\right)\right|\\
			&\quad+\left|\E\left( \Delta \mathbf{1}(-a\le \Delta<0)\int_{-\Delta}^{0}\left(\mathbf{1}(W \le z-t) - \mathbf{1}(W \le z)\right)\dx t \right)\right|\\
			&\le \E\left( \Delta \mathbf{1}(0\le \Delta\le a)\int_{-\Delta}^{0}\mathbf{1}(z-a\le  W \le z+a)\dx t\right) \\
			&\quad+\E\left(-\Delta \mathbf{1}(-a\le \Delta<0)\int_{0}^{-\Delta}\mathbf{1}(z-a\le  W \le z+a)\dx t \right)\\
			&= \E\left( \Delta^2 \mathbf{1}(|\Delta|\le a)\mathbf{1}(z-a\le  W \le z+a)\right)\\
			&\le \frac{C(p)(1+\E|W|^{2p})a(\tau+\sqrt{\E R^{2}})}{(1+z)^p},
		\end{aligned} 
	\end{equation}
	where we have applied Lemma \ref{lem.concentration.inequality} in the last inequality.
	Combining \eqref{b032}, \eqref{b0321} and \eqref{b0324} yields
	\begin{equation} \label{b0325}
		\begin{aligned}
			|T_{3,2}|&\le \frac{C(p)(1+\E|W|^{2p})(a^3+a(\tau+\sqrt{\E R^{2}}))}{(1+z)^p}.
		\end{aligned} 
	\end{equation}
	Combining \eqref{s5}, \eqref{b031} and \eqref{b0325} yields
	\begin{align} \label{b03}
		\begin{aligned}
			|T_{3}| &\le \dfrac{C(p)(1+\E|W|^{2p})}{(1+z)^p}\left(\dfrac{a \sqrt{\E R^2}}{\tau}+ a + \dfrac{ a^3}{\tau}+\dfrac{\sqrt{\E\Delta^4 \mathbf{1}(|\Delta|>a)}}{\tau}\right). 
		\end{aligned} 
	\end{align}
	The conclusion of the theorem follows from \eqref{s4}, \eqref{b01}, \eqref{b02} and \eqref{b03}.
\end{proof}

\section{Applications}\label{sec.appl}
In this section, we apply Theorem \ref{thm.main11} to obtain non-uniform Berry--Esseen bounds in the central limit theorem
for the squared-length of the total spin in the mean-field classical $N$-vector models and Jack deformations of the character ratio.

\subsection{Mean-field  classical $N$-vector models}\label{subsec.meanfield}

Let $N\ge 2$ be an integer, let ${\mathbb{S}}^{N-1}$ denote the unit sphere in ${\mathbb{R}}^N$ and let $\mu$ be the uniform
probability measure on ${\mathbb{S}}^{N-1}$.
We consider the mean-field classical $N$-vector spin models (also called the mean-field $O(N)$ models),
where each spin $\sigma_i$ is in ${\mathbb{S}}^{N-1}$, at a complete graph vertex $i$ among $n$ vertices (see \cite{kirkpatrick2016asymptotics}).
The state space is $\Omega_n=({\mathbb{S}}^{N-1})^n$ with the product measure $\mathbb{P}_n=\mu\times\dots\times \mu$. 
In the absence of an external field, each spin configuration $\sigma=(\sigma_1,\dots, \sigma_n)$ in the state space $\Omega_n$ has a Hamiltonian defined by
\[H_n(\sigma)=-\dfrac{1}{2n}\sum_{i=1}^n \sum_{j=1}^n \langle \sigma_i , \sigma_j \rangle,\]
where $\langle \cdot , \cdot \rangle$ is the inner product in ${\mathbb{R}}^N$.
Let $\beta>0$ be the inverse temperature. 
The Gibbs measure with Hamiltonian $H_n$ is the probability measure $\mathbb{P}_{n,\beta}$ on $\Omega_n$
with density function:
\[\dx\mathbb{P}_{n,\beta}(\sigma)=\dfrac{1}{Z_{n,\beta}}\exp\left( -\beta H_{n}(\sigma)\right)\dx\mathbb{P}_n(\sigma),\]
where $Z_{n,\beta}=\int_{\Omega_n}\exp\left(-\beta H_n (\sigma) \right)\dx \mathbb{P}_n(\sigma).$
The cases where $N=2$, $N=3$, and $N=4$ of this model reduce to the $XY$ model, the Heisenberg model, and the Toy model, respectively 
(see, e.g., \cite[p. 412]{friedli2017statistical}). 

Kirkpatrick and Nawaz \cite{kirkpatrick2016asymptotics} proved a multivariate central limit theorem
for the total spin in the subcritical temperature, a non-normal limit theorem
for the total spin at the critical temperature, and a central limit theorem
for the squared-length of the total spin in the supercritical temperature. 
In particular, Kirkpatrick and Nawaz \cite{kirkpatrick2016asymptotics} proved
that when $\beta >N$,
the bounded Lipschitz distance between the squared-length of the total spin
and a standard normal random variable is bounded by $C (\log n/n)^{1/4}$.
Thành and Tu (see \cite[Theorem 1.1]{thanh2019error}) improved the result 
of Kirkpatrick and Nawaz \cite{kirkpatrick2016asymptotics} by proving a uniform Berry--Esseen inequality with bound $n^{-1/2}$.
The special case $N=3$ (the Heisenberg model) was studied by Shao and Zhang \cite[Theorem 3.3]{shao2019berry}.
In this subsection, we will apply Theorem \ref{thm.main11} to give a non-uniform Berry--Esseen bound, refining the result of 
Th\`{a}nh
and Tu \cite{thanh2019error}.

Let $I_{\nu}$ denote the modified Bessel function of the first kind
(see, e.g., \cite[p. 713]{arfken2005mathematical})
which is defined by
\[I_{\nu}(x) = \sum_{s=0}^{\infty} \frac{1}{s!\Gamma(s+\nu+1)}\left(\frac{x}{2}\right)^{\nu+2 s},\]
where $\Gamma(\cdot)$ is the Gamma function. Let  
\begin{equation*}
	f(x)=\dfrac{I_{N/2}(x)}{I_{(N/2)-1}(x)},\ x>0,
\end{equation*}
and
\begin{equation}\label{define.B2}
	B^2=\dfrac{4\beta^2}{(1-\beta f^{'}(b))b^2}\left[1-\dfrac{(N-1)f(b)}{b}-(f(b))^2 \right].
\end{equation}
In the case $\beta>N$, there is a unique strictly positive solution $b$ to
the equation (see, e.g., \cite[p. 2]{thanh2019error}):
\begin{equation*}
	x - \beta f(x)= 0.
\end{equation*}
Let $S_n=\sum_{j=1}^n \sigma_j$ and
\begin{equation}\label{define.W}
	W_{n}= \sqrt{n}\left(\dfrac{\beta^2}{n^2 b^2}|S_n|^2 -1 \right).
\end{equation}
Here and hereafter, for vector $x\in\R^N$, the Euclidean norm of $x$ is also denoted by $|x|$.
The result of this subsection is the following theorem. This result extends Theorem 1.1 of Th\`{a}nh and Tu \cite{thanh2019error} and, 
consequently, also extends Theorem 3.3 of Shao and Zhang \cite{shao2019berry}.

\begin{theorem} \label{app2}
	Let $p\ge 1$ and let $W_n$ be the random variable defined as in \eqref{define.W} and $B^2$ as in \eqref{define.B2} with $\beta >N$. Then for any $z\in \mathbb{R}$, we have
	\begin{equation}\label{ap2} 
		\left|\mathbb{P}\left(W_n/B\le z\right) -\Phi(z)\right| \le \dfrac{C(N,\beta,p)}{(1+|z|)^p n^{1/2}}.
	\end{equation}
\end{theorem}
\begin{proof}
	We describe the construction of an exchangeable pair as follows.
	Let $\sigma' = \{\sigma_1',\ldots,\sigma_n'\}$, 
	where for each $i$ fixed, $\sigma_i'$ is an independent copy of $\sigma_i$ given $\{\sigma_j, j\ne i\}$, i.e., given $\{\sigma_j, j\ne i\}$,
	$\sigma_i'$ and $\sigma_i$ have the same distribution and $\sigma_i'$ is conditionally independent of $\sigma_i$.
	Let $I$ be a random index independent of all others and uniformly distributed over $\{1,\ldots,n\}$, and let
	\begin{align} \label{define.Wprime}
		W_{n}' = \sqrt{n}\left(\dfrac{\beta^2}{n^2 b^2}|S_n'|^2 - 1 \right),
	\end{align}
	where  $S_n' = \sum_{j=1}^n\sigma_j-\sigma_I + \sigma_I'$. 
	Then $(W_{n},W_{n}')$ is an exchangeable pair \cite[p. 1124]{kirkpatrick2016asymptotics}.
	Since $|S_n|\le n$, $|S_{n}'|\le n$ and $|S_n-S_{n}'|=|\sigma_I - \sigma_{I}'|\le 2$, it follows from \eqref{define.W} and \eqref{define.Wprime} that
	\begin{equation}\label{bounded-pair-meanfield}
		|\Delta|:=|W_{n}-W_{n}'|=\sqrt{n}\left|\dfrac{\beta^2}{n^2 b^2}(S_n-S_n')(S_n+S_n')\right|\le \dfrac{4\beta^2}{\sqrt{n} b^2}.
	\end{equation}
	For $1\le i\le n$, set
	\[\sigma^{(i)}=S_n-\sigma_i.\]
	Kirkpatrick and Nawaz \cite[Equation (8)]{kirkpatrick2016asymptotics} proved that
	\begin{equation}\label{kk01}
		\E(W_n-W_{n}^{'}|\sigma) = \dfrac{2}{n}W_n +\dfrac{2}{\sqrt{n}}-\dfrac{2\beta^2}{n^{3/2}b^2}-\dfrac{2\beta^2}{n^{5/2}b^2}\sum_{i=1}^{n}f\left(\dfrac{\beta|\sigma^{(i)}|}{n}\right)|\sigma^{(i)}|.
	\end{equation}
	Let
	\[R_1=-\dfrac{2\beta^2}{n^{3/2}b^2}+\dfrac{2\beta^2}{n^{5/2}b^2}\sum_{i=1}^{n}\left(f\left(\dfrac{\beta|S_n|}{n}\right)|S_n|-f\left(\dfrac{\beta|\sigma^{(i)}|}{n}\right)|\sigma^{(i)}|\right).\]
	Then, it follows from \eqref{kk01} that
	\begin{equation}\label{kk04}
		\E(W_n-W_{n}^{'}|\sigma) = \dfrac{2}{n}W_n +\dfrac{2}{\sqrt{n}}-\dfrac{2\beta}{n^{1/2}b^2}\left(\dfrac{\beta|S_n|}{n}\right)f\left(\dfrac{\beta|S_n|}{n}\right)+R_1.
	\end{equation}
	Let $g(x)=xf(x)$, $h(x)=f(x)/x$, $x>0$, and let
	\[V=\dfrac{\beta|S_n|}{nb}+1\in[1,1+\beta/b].\]
	By Taylor's expansion and definition of $W_n$, we have (see Th\`{a}nh and Tu \cite[Equations (2.2) and (2.3)]{thanh2019error})
	for some
	positive random variable $\xi$,
	\begin{equation}\label{eq:bound02}
		g\left(\dfrac{\beta|S_n|}{n}\right) =
		g(b) + g'(b)\left(\dfrac{\beta|S_n|}{n}-b\right) +\dfrac{g''(\xi)}{2} \left(\dfrac{\beta|S_n|}{n}-b\right)^2
	\end{equation}
	and	
	\begin{equation}\label{eq:bound03}
		\begin{split}
			\dfrac{\beta|S_n|}{n}-b&=\dfrac{bW_n}{\sqrt{n}V}=\dfrac{bW_n}{2\sqrt{n}}-\dfrac{bW_{n}^2}{2nV^2}.
		\end{split}
	\end{equation}
	By using \eqref{kk04}--\eqref{eq:bound03} and noting that $b=\beta f(b)$, we obtain
	\begin{equation}\label{nonlinear-meanfield}
		\begin{split}
			&\E(W_n-W_{n}^{'}|\sigma)\\
			&= \dfrac{2W_n}{n} +\dfrac{2}{\sqrt{n}} +R_1 - \dfrac{2\beta}{n^{1/2}b^2} \left(g(b) 
			+ g'(b)\left( \dfrac{bW_n}{2\sqrt{n}}-\dfrac{bW_{n}^2}{2nV^2}\right)+\dfrac{g''(\xi)}{2} \left(\dfrac{\beta|S_n|}{n}-b\right)^2\right) \\
			&= \dfrac{2W_n}{n} +\dfrac{2}{\sqrt{n}} +R_1-\dfrac{2\beta}{n^{1/2}b^2} \left(\dfrac{b^2}{\beta} + \left(\dfrac{b}{\beta}+bf'(b) \right)\left( \dfrac{bW_n}{2\sqrt{n}}-\dfrac{bW_{n}^2}{2nV^2}\right)+\dfrac{g''(\xi)b^2W_{n}^2}{2nV^2}\right) \\
			&= \dfrac{1-\beta f'(b)}{n}(W_n+R^*),
		\end{split}
	\end{equation}
	where
	\begin{equation}\label{define.R}
		R^*=\dfrac{n}{1-\beta f'(b)}\left(R_1+\dfrac{\beta W_{n}^2}{n^{3/2 }V^2} \left(\dfrac{1}{\beta}+f'(b) -g''(\xi)\right)\right).
	\end{equation}
	Since the $\sigma$-field generated by $\sigma$ contains the $\sigma$-field generated by $W_n$, \eqref{nonlinear-meanfield} implies that 
	the exchangeable pair $(W_{n},W_{n}')$ satisfies \eqref{nonlinear} with
	\[\tau=\dfrac{1-\beta f'(b)}{n}\ \text{ and }\ R=\dfrac{(\beta f'(b)-1)\E(R^*|W_n)}{n}.\]
	Th\`{a}nh and Tu \cite{thanh2019error} (see Lemma A.2 and Equation (A.5) ibidem) also showed that for all $x>0$,
	\begin{equation}\label{f-and-g03}
		|f(x)|<1,\ 0<f'(x)<1,\ \ |g''(x)|<6\ \text{ and } -5\le h'(x)<0.
	\end{equation}
	By noting that 
	\begin{equation}\label{norm-of-Sn}
		\left| |S_n|-|\sigma^{(i)}|\right|\le |\sigma_i|=1,\ 1\le i\le n,
	\end{equation}
	and using \eqref{f-and-g03}, we have
	\begin{equation}\label{estimate-R1}
		\begin{aligned}
			|R_1|&\le  \dfrac{2\beta^2}{n^{3/2}b^2}+\dfrac{2\beta^2}{n^{5/2}b^2}\left|\sum_{i=1}^{n}\left(f\left(\dfrac{\beta|S_n|}{n}\right)|S_n|-f\left(\dfrac{\beta|\sigma^{(i)}|}{n}\right)|\sigma^{(i)}|\right)\right|\\
			&\le \dfrac{2\beta^2}{n^{3/2}b^2}+\dfrac{2\beta^2}{n^{5/2}b^2}\left|\sum_{i=1}^{n}|S_n|\left(f\left(\dfrac{\beta|S_n|}{n}\right)-f\left(\dfrac{\beta|\sigma^{(i)}|}{n}\right)\right)\right|\\
			&\qquad+\dfrac{2\beta^2}{n^{5/2}b^2}\left|\sum_{i=1}^{n}f\left(\dfrac{\beta|\sigma^{(i)}|}{n}\right) \left(|S_n|-|\sigma^{(i)}|\right)\right|\\
			&\le \dfrac{2\beta^2}{n^{3/2}b^2}+\dfrac{2\beta^2}{n^{5/2}b^2}\sum_{i=1}^{n}\| f'\|\dfrac{\beta|S_n|}{n} \left| |S_n|-|\sigma^{(i)}|\right|+\dfrac{2\beta^2}{n^{5/2}b^2}\sum_{i=1}^{n}\|f\||\sigma_i|\\
			&\le \dfrac{2\beta^2}{n^{3/2}b^2}\left(1+\beta+1\right)\le \dfrac{C(\beta)}{n^{3/2}}.
		\end{aligned}
	\end{equation}
	Th\`{a}nh and Tu \cite[Lemma A.1]{thanh2019error} proved that
	\[\E\left(\dfrac{\beta|S_n|}{n} - b\right)^2 \le \dfrac{C(\beta)}{n}.\]
	Following the same steps as in the proof of Lemma A.1 by Th\`{a}nh and Tu \cite{thanh2019error}, we also have
	\[\E\Big|\dfrac{\beta|S_n|}{n} - b\Big|^{2p} \le \dfrac{C(p,\beta)}{n^{p}}.\]
	It thus follows that
	\begin{equation}\label{moment.W}
		\begin{aligned}
			\E|W_{n}|^{2p} &= n^{p} \E\left|\dfrac{\beta^2 |S_n|^2}{n^2 b^2} - 1 \right|^{2p} \\
			&= n^p \E\left|\left(\dfrac{\beta|S_n|}{nb}+1\right)\left(\dfrac{\beta |S_n|}{n b} - 1\right) \right|^{2p}\\
			& \le n^p\left(1+\dfrac{\beta}{b}\right)^{2p} \E\left|\dfrac{\beta |S_n|}{n b} - 1 \right|^{2p}\\
			& \le C(p,\beta).
		\end{aligned}	
	\end{equation}
	By combining \eqref{define.R}, \eqref{f-and-g03}, \eqref{estimate-R1} and \eqref{moment.W}, we obtain
	\begin{equation*}
		\begin{aligned}
			\E|R^*|^2&\le \dfrac{C(\beta)}{n}\left(1+\E W_{n}^4\right)\le \dfrac{C(\beta)}{n},
		\end{aligned}
	\end{equation*}
	which, together with Jensen's inequality for conditional expectation, implies
	\begin{equation}\label{estimate-moment-R*}
		\begin{aligned}
			\E\left(\E(R^*|W_n)\right)^2\le \E\left(\E((R^*)^2|W_n)\right)=\E|R^*|^2\le  \dfrac{C(\beta)}{n}.
		\end{aligned}
	\end{equation}
	For $1\le i\le n$, set
	\[b_i=\beta|\sigma^{(i)}|/n.\]
	Th\`{a}nh and Tu \cite{thanh2019error} (see Equation (2.4) and the computations on page 7 ibidem) showed that
	\begin{equation}\label{estimate.conditional.variance01} 
		\frac{1}{2\tau}\E((W_{n}-W_{n}')^2|\sigma)-B^2
		= \dfrac{2\beta^4}{n^3 b^4(1-\beta f'(b))} \left(R_2-\dfrac{2b}{\beta} R_3+R_4+R_{5}\right),
	\end{equation}
	where 
	\begin{align*}
		& R_2 = \sum_{i=1}^n\left(1-\dfrac{N-1}{\beta}\right)\left(|\sigma^{(i)}|^2 - \dfrac{(n-1)^2 b^2}{\beta^2} \right), \\
		& R_3 = \sum_{i=1}^n \left(|\sigma^{(i)}|\langle \sigma_i,\sigma^{(i)}\rangle - \dfrac{n^2 b^3}{\beta^3}\right),\\
		& R_4 = \sum_{i=1}^n \left(\langle \sigma_i,\sigma^{(i)}\rangle^2 - \left(1-\dfrac{N-1}{\beta}\right)\dfrac{(n-1)^2 b^2}{\beta^2}\right), \\
		& R_{5} = \sum_{i=1}^n(1-N)\left(\dfrac{f(b_i)}{b_i}- \dfrac{1}{\beta} \right)|\sigma^{(i)}|^2-2\sum_{i=1}^n\left(f(b_i)- \dfrac{b}{\beta} \right) |\sigma^{(i)}|\langle \sigma^{(i)},\sigma_i\rangle.
	\end{align*}
	In order to apply Theorem \ref{thm.main11}, we need to bound
	\[\E\left(\frac{1}{2\tau}\E((W_{n}-W_{n}')^2|W_{n})-B^2\right)^2,\]
	which is more challenging than bounding $\E\left|\frac{1}{2\tau}\E((W_{n}-W_{n}')^2|W_{n})-B^2\right|$ as in Th\`{a}nh and Tu \cite{thanh2019error}.
	By using \eqref{define.W}
	and \eqref{norm-of-Sn}, we have
	\begin{equation}\label{estimate-R2}
		\begin{aligned}
			|R_2|&=\left(1-\dfrac{N-1}{\beta}\right)\sum_{i=1}^n\left(\dfrac{n^2 b^2}{\beta^2}\left(\dfrac{\beta^2}{n^2 b^2}|S_n|^2-1\right)-\left(|S_n|^2-|\sigma^{(i)}|^2\right)+ \dfrac{n^2b^2-(n-1)^2b^2}{\beta^2} \right)\\
			&\le C(N,\beta)\left(n^3\left(\dfrac{\beta^2}{n^2 b^2}|S_n|^2-1\right)+n^2\left(|S_n|-|\sigma^{(i)}|\right)+\dfrac{n(2n-1)b^2}{\beta^2}\right)\\
			&\le C(N,\beta)n^{5/2}\left(|W_n|+1\right).
		\end{aligned}
	\end{equation}
	Similarly,
	\begin{equation}\label{estimate-R3}
		\begin{aligned}
			|R_3|&=\left|\sum_{i=1}^n\left(|S_n|\langle \sigma_i,S_n\rangle- \dfrac{n^2 b^3}{\beta^3}+|\sigma^{(i)}|\langle \sigma_i,\sigma^{(i)}\rangle -|S_n|\langle \sigma_i,S_n\rangle \right)\right|\\
			&=\left|\left(|S_n|^3- \dfrac{n^3 b^3}{\beta^3}\right)+\sum_{i=1}^n \left( (|\sigma^{(i)}|-|S_n|)\langle \sigma_i,\sigma^{(i)}\rangle -|S_n|\langle \sigma_i,\sigma_i\rangle \right)\right|\\
			&\le \left|\left(|S_n|^3- \dfrac{n^3 b^3}{\beta^3}\right)\right|+\sum_{i=1}^n |\langle \sigma_i,\sigma^{(i)}\rangle| +|S_n|\sum_{i=1}^n \langle \sigma_i,\sigma_i\rangle\\
			&\le C(\beta)\left(n^2\left||S_n|- \dfrac{n b}{\beta}\right|+n^2\right)\\
			&\le C(\beta)\left(n^3\left|\dfrac{\beta|S_n|}{nb}- 1\right|+n^2\right)\\
			&\le C(\beta)n^{5/2}\left(|W_n|+1\right),
		\end{aligned}
	\end{equation}
	and
	\begin{equation}\label{estimate-R5}
		\begin{aligned}
			|R_5|&\le \left|\sum_{i=1}^n(1-N)\left(\dfrac{f(b_i)}{b_i}- \dfrac{1}{\beta} \right)|\sigma^{(i)}|^2\right|+2\left|\sum_{i=1}^n\left(f(b_i)- \dfrac{b}{\beta} \right) |\sigma^{(i)}|\langle \sigma^{(i)},\sigma_i\rangle\right|\\
			&=\left|\sum_{i=1}^n(1-N)\left(\dfrac{f(b_i)}{b_i}- \dfrac{f(b)}{b} \right)|\sigma^{(i)}|^2\right|+2\left|\sum_{i=1}^n\left(f(b_i)-f(b)\right) |\sigma^{(i)}|\langle \sigma^{(i)},\sigma_i\rangle\right|\\
			&\le C(N,\beta)\left(\|g\|+\|f\|\right)n^2\sum_{i=1}^n|b_i-b|\\
			&= C(N,\beta)\left(\|g\|+\|f\|\right)n^2\sum_{i=1}^n\left(\left(\dfrac{\beta|S_n|}{n}- b\right)+\dfrac{\beta}{n}(\sigma^{(i)}-|S_n|)\right)\\
			&\le  C(N,\beta)\left( n^3\left|\dfrac{\beta|S_n|}{nb}- 1\right| +\sum_{i=1}^n\dfrac{\beta}{n}\left||\sigma^{(i)}|-|S_n|\right|\right)\le C(N,\beta)n^{5/2}\left(|W_n|+1\right),
		\end{aligned}
	\end{equation}
	where in \eqref{estimate-R5}, we have also used the facts that $1/\beta=f(b)/b$ and $b_i=\beta|\sigma^{(i)}|/n$. By using \eqref{moment.W}
	and \eqref{estimate-R2}--\eqref{estimate-R5}, we have
	\begin{equation}\label{estimate.conditional.variance03} 
		\E(R_{2}^2+R_{3}^2+R_{5}^2)\le C(N,\beta)n^5.
	\end{equation}
	Th\`{a}nh and Tu \cite[p. 9]{thanh2019error} proved that
	\begin{equation}\label{estimate.conditional.variance05} 
		\E R_{4}^2\le C(N,\beta)n^5.
	\end{equation}
	Combining \eqref{estimate.conditional.variance01}, \eqref{estimate.conditional.variance03} and \eqref{estimate.conditional.variance05} yields
	\begin{equation}\label{estimate.conditional.variance06}
		\E\left(\frac{1}{2\tau}\E((W_{n}/B-W_{n}'/B)^2|\sigma)-1\right)^2\le \dfrac{C(N,\beta)}{n}.
	\end{equation}
	Since the $\sigma$-field generated by $\sigma$ contains the $\sigma$-field generated by $W_n$, we obtain from elementary computations and \eqref{estimate.conditional.variance06} that
	\begin{equation}\label{estimate.conditional.variance07}
		\begin{split}
			\E\left(\frac{1}{2\tau}\E((W_{n}/B-W_{n}'/B)^2|W_{n})-1\right)^2&\le \E\left(\frac{1}{2\tau}\E((W_{n}/B-W_{n}'/B)^2|\sigma)-1\right)^2\\
			&\le \dfrac{C(N,\beta)}{n}.
		\end{split}
	\end{equation}
	For all $x>0$, we have $0<f'(x)<f(x)/x$ (see, e.g., \cite[Equation (A.4)]{thanh2019error}). Since $1/\beta=f(b)/b$, elementary computations show
	$0<1-\beta f'(b)<1$.
	From \eqref{bounded-pair-meanfield} and \eqref{nonlinear-meanfield}, we are able to apply 
	Theorem \ref{thm.main11} with 
	\begin{equation}\label{define-a-tau}
		W=\dfrac{W_n}{B},\ a= \dfrac{4\beta^2}{\sqrt{n} b^2},\ \tau=\dfrac{1-\beta f'(b)}{n}, \text{ and }R=\dfrac{(\beta f'(b)-1)\E(R^{*}|W_n)}{n}.
	\end{equation}
	Using \eqref{estimate-moment-R*} and \eqref{define-a-tau}, we obtain
	\begin{equation}\label{estimate-moment-R}
		\begin{aligned}
			\dfrac{\sqrt{\E R^2}}{\tau}&\le \dfrac{C(\beta)}{\sqrt{n}} \text{ and } a+\dfrac{a^3}{\tau}\le \dfrac{C(\beta)}{\sqrt{n}}. 
		\end{aligned}
	\end{equation}
	By collecting the bounds \eqref{moment.W}, \eqref{estimate.conditional.variance07} and \eqref{estimate-moment-R}, we obtain \eqref{ap2}.
	The proof of the theorem is completed.
\end{proof}

\subsection{Jack Measures}\label{subsec.jack}

In this subsection, we apply Theorem \ref{thm.main11} to obtain a non-uniform Berry--Esseen bound for Jack$_{\alpha}$ measure.
Let $\mathcal{P}_n$ be the set of all partitions of positive integer $n$ and let $\alpha>0$. For each $\lambda\in \mathcal{P}_n$,
we define Jack$_\alpha$ measure of $\lambda$ to be
$$\mathbb{P}_{\alpha}(\lambda)=\dfrac{\alpha^n n!}{\Pi_{s\in \lambda}(\alpha a(s)+l(s)+1)(\alpha a(s)+l(s)+\alpha)},$$
where the product is over all boxes in the partition. Here $a(s)$ denotes the number of boxes in the
same row of $s$ and to the right of $s$ (the ``arm'' of $s$) and $l(s)$
denotes the number of boxes in the same column of $s$ and below $s$ (the ``leg'' of $s$).
For example, the top-left box (the northwest corner box) $s$ of the partition of $4$
\begin{equation*}
	\lambda=
	\begin{array}{lrc}
		\Box\  \ \Box\  \ \Box\\
		\Box
	\end{array}
\end{equation*}
has $a(s)=2$ and $l(s)=1$. This partition
has Jack$_\alpha$ measure
$$\mathbb{P}_{\alpha}(\lambda)=\dfrac{6\alpha}{(3\alpha+1)(\alpha+1)^2}.$$
It can be verified that Jack$_\alpha$ measure is a probability measure on $\mathcal{P}_n$ (see Remark \ref{rem.Jackmeasure}).
This measure was introduced by Kerov \cite{kerov2000anisotropic} and has interesting connections to the Gaussian $\beta$-Ensemble
in the random matrix theory with $\beta=2/\alpha$. This research direction has recently been explored in several papers
(see \cite{borodin2005z,dolega2016gaussian,fulman2004stein,guionnet2019rigidity,huang2021law}
and the references therein). Jack$_\alpha$ measure was also presented as an important area of research in a survey paper on random partitions 
by Okounkov \cite{okounkov2006uses}. For the case where $\alpha=1$, Jack$_{\alpha}$ measure agrees with the Plancherel measure on the set of all the irreducible representations of the symmetric group $S_n$.

Let $\alpha>0$ and
\begin{equation}\label{fulman30}
	W_{n,\alpha}=W_{n,\alpha}(\lambda)=\dfrac{\sum_{i}\Big(\alpha{{\lambda_i}
			\choose 2}-{{\lambda_{i}^{'}} \choose 2}\Big)}{\sqrt{\alpha{n
				\choose 2}}},
\end{equation} 
where $\lambda$ is chosen from the Jack$_\alpha$
measure on partitions of size $n$, $\lambda_i$ is the length of the
$i$-th row of $\lambda$ and $\lambda_{i}^{'}$ is the length of the
$i$-th column of $\lambda$. When $\alpha=1$, $W_{n,1}$ coincides with the character ratio
of the symmetric group on transpositions (see, e.g., Chen et al. \cite[p. 444]{chen2021error}).
Fulman \cite{fulman2004stein} used Stein's method for exchangeable pairs and  proved that
\begin{equation}\label{fulman31}
	\sup_{x\in {\mathbb{R}}}|\mathbb{P}_{\alpha}(W_{n,\alpha}\le
	x)-\Phi(x)|\le \dfrac{C(\alpha)}{n^{1/4}}. 
\end{equation}
The rate in \eqref{fulman31} was sharpened by 
Fulman \cite{fulman2006inductive} to $C(\alpha)n^{-1/2}$ using the inductive approach to Stein's method. 
We note that in all these results, $\alpha>0$ is fixed, but we do not know how $C(\alpha)$ depends on $\alpha$. 
An explicit constant is obtained by
Shao and Su \cite{shao2006berry} only when $\alpha=1$. 
Fulman \cite{fulman2004stein} conjectured that for general $\alpha\ge 1$, the correct bound
for the Kolmogorov distance is a universal constant multiplied by $\max\left\{1/\sqrt{n},\sqrt{\alpha}/n\right\}$.
Using Stein's method and zero-bias couplings, Fulman and 
Goldstein \cite{fulman2011zero} proved that it is the correct bound for the Wasserstein distance. 
More recently, Chen et al. \cite{chen2021error} achieved the rate conjectured by Fulman up to a $\log n$ factor.
When $\alpha$ is fixed, Chen et al. \cite{chen2021error}
confirmed Fulman's conjecture by proving that
\begin{equation*}
	\sup_{x\in {\mathbb{R}}}|\mathbb{P}_{\alpha}(W_{n,\alpha}\le
	x)-\Phi(x)|\le \dfrac{8.2}{\sqrt{n}}.
\end{equation*} 

As remarked by Fulman \cite{fulman2004stein}, one is usually interested in 
$\alpha$ being fixed, as $\alpha$ is a parameter which represents
the symmetries of the system. In this subsection, we consider the case
$\log^2n/n \le \alpha \le n/\log^2n$ and provide a non-uniform Berry--Esseen bound for
$W_{n,\alpha}$ by applying Theorem \ref{thm.main11}. The same result was also achieved by Chen et al. \cite{chen2021error} using Stein's method for zero-biased couplings.

\begin{theorem}\label{thm.Jack} 
	Let $p\ge2$, $n\ge 3$, $\log^2n/n\le \alpha\le n/\log^2n$ and $W_{n,\alpha}$ in \eqref{fulman30}. 
	Then for all  $z\in {\mathbb{R}}$, we have
	\begin{equation}\label{Fu01}
		\left|{\mathbb{P}}_{\alpha}(W_{n,\alpha}\le
		z)-\Phi(z)\right|\le \dfrac{C(p)}{(1+|z|)^{p}\sqrt{n}}.
	\end{equation}
\end{theorem}

\begin{proof}
	Firstly, consider the case where $1\le \alpha\le n/\log^2 n$. Kerov \cite{kerov2000anisotropic}
	proved that there is a growth process giving a sequence of partitions		
	$(\lambda(1),\dots,\lambda(n))$ with partition $\lambda(j)$ distributed
	according to the Jack$_\alpha$ measure on the set of partitions of size $j$.
	Given the Kerov growth process, let
	$X_{1,\alpha}=0$, $X_{j,\alpha}=c_{\alpha}(x)$ where $x$ is the box
	added to $\lambda(j-1)$ to obtain $\lambda(j)$ and the
	``$\alpha$-content'' $c_\alpha(x)$ of a box $x$ is defined to be
	$\alpha(\text{column number of }x-1)- (\text{row number of }x-1)$, $j\ge 2$. 
	Then, from Fulman \cite{fulman2006inductive}, one can write
	\begin{equation*}
		W_{n,\alpha}=\dfrac{\sum_{j=1}^n X_{j,\alpha}}{\sqrt{\alpha{n\choose 2}}}.
	\end{equation*} 
	Therefore, constructing $\nu$ from the
	Jack$_\alpha$ measure on partitions of $n-1$ and then taking
	one step in Kerov's growth process yields $\lambda$ with the
	Jack$_\alpha$ measure on partitions of $n$, we have 
	\begin{equation*}\label{FG6}
		W_{n,\alpha}=V_{n,\alpha}+\eta_{n,\alpha},
	\end{equation*} where
	\begin{equation*}\label{FG7}
		V_{n,\alpha}=\dfrac{\sum_{x\in\nu}c_{\alpha}(x)}{\sqrt{\alpha{n\choose
					2}}}= \sqrt{\dfrac{n-2}{n}} W_{n-1,\alpha},\text{  }
		\eta_{n,\alpha}=\dfrac{X_{n,\alpha}}{\sqrt{\alpha{n \choose
					2}}}=\dfrac{c_{\alpha}(\lambda/\nu)}{\sqrt{\alpha{n \choose 2}}},
	\end{equation*}
	and $c_{\alpha}(\lambda/\nu)$ denotes the $\alpha$-content of the box
	added to $\nu$ to obtain $\lambda$.
	As in Fulman \cite{fulman2004stein}, we construct $\lambda'$ by taking another step in the Kerov growth process from $\nu$, independently of $\lambda/\nu$ given $\nu$, 
	and then forming $W_{n,\alpha}'$ from $\lambda'$ as $W_{n,\alpha}$ is formed from $\lambda$, that is
	\begin{equation*}\label{FG8}
		W_{n,\alpha}'=V_{n,\alpha}+\eta_{n,\alpha}',
	\end{equation*} where
	\begin{equation*}\label{FG9}
		\eta_{n,\alpha}'=\dfrac{X_{n,\alpha}'}{\sqrt{\alpha{n \choose
					2}}}=\dfrac{c_{\alpha}(\lambda'/\nu)}{\sqrt{\alpha{n \choose 2}}},
	\end{equation*} 
	and $c_{\alpha}(\lambda'/\nu)$ denotes the $\alpha$-content of the box
	added to $\nu$ to obtain $\lambda'$. 
	Fulman \cite[Proposition 6.1]{fulman2004stein} showed that
	$(W_{n,\alpha}, W_{n,\alpha}')$ is an exchangeable pair
	that satisfies \eqref{linear} with $\tau=2/n$. Let
	\[\Delta_{n,\alpha}=W_{n,\alpha}-W_{n,\alpha}'=\eta_{n,\alpha}-\eta_{n,\alpha}'.\]
	For all $q>1$, Chen et al. \cite{chen2021error}
	(see (4.20) ibidem) proved that
	\begin{equation}\label{Fu70}
		{\mathbb{P}}_{\alpha}\left(|\eta_{n,\alpha}|> \dfrac{e\sqrt{2q}}{\sqrt{n-1}}\right)\le  \dfrac{\alpha}{\pi(q-1)q^{e\sqrt{qn/\alpha}}}.
	\end{equation}
	Similarly,
	\begin{equation}\label{Fu71}
		{\mathbb{P}}_{\alpha}\left(|\eta_{n,\alpha}'|> \dfrac{e\sqrt{2q}}{\sqrt{n-1}}\right)\le  \dfrac{\alpha}{\pi(q-1)q^{e\sqrt{qn/\alpha}}}.
	\end{equation}
	Combining \eqref{Fu70} and \eqref{Fu71} yields
	\begin{equation}\label{Fu73}
		\begin{aligned}
			{\mathbb{P}}_{\alpha}\left(|\Delta_{n,\alpha}|>\dfrac{2e\sqrt{2q}}{\sqrt{n-1}}\right)&\le {\mathbb{P}}_{\alpha}\left(|\eta_{n,\alpha}'|> \dfrac{e\sqrt{2q}}{\sqrt{n-1}}\right)+	{\mathbb{P}}_{\alpha}\left(|\eta_{n,\alpha}'|> \dfrac{e\sqrt{2q}}{\sqrt{n-1}}\right)\\
			&\le  \dfrac{2\alpha}{\pi(q-1)q^{e\sqrt{qn/\alpha}}}.
		\end{aligned}
	\end{equation}
	Let
	\[q=(p+e)^2,\ a=\dfrac{2e\sqrt{2q}}{\sqrt{n-1}}.\]
	By applying \eqref{Fu73} with noting that $n/\alpha\ge \log^2n$, we have
	\begin{equation}\label{Fu77}
		\begin{aligned}
			{\mathbb{P}}_{\alpha}\left(|\Delta_{n,\alpha}|>a\right)
			&\le \dfrac{\alpha}{e^{2e(p+e)\log n}}= \dfrac{\alpha}{n^{2e(p+e)}}.
		\end{aligned}
	\end{equation}
	It is clear from the definition of the content of a box that 
	\begin{equation}\label{Fu78}
		|c_{\alpha}(\lambda/\nu)|\le \max \left\{\alpha(\lambda_{1}-1),(\lambda_{1}'-1)\right\},
	\end{equation}
	where $\lambda_{1}$ and $\lambda_{1}'$ denote the length of the first row
	and the length of the first column of $\lambda$, respectively. 
	By \eqref{Fu78}, we have $|c_{\alpha}(\lambda/\nu)|\le \alpha n$. Similarly, $|c_{\alpha}(\lambda'/\nu)|\le \alpha n$.
	Therefore
	\begin{equation}\label{Fu80}
		\begin{aligned}
			|\Delta_{n,\alpha}|&\le \dfrac{c_{\alpha}(\lambda/\nu)}{\sqrt{\alpha{n \choose 2}}}+\dfrac{c_{\alpha}(\lambda'/\nu)}{\sqrt{\alpha{n \choose 2}}}\le 4\sqrt{\alpha}.
		\end{aligned}
	\end{equation}
	Combining \eqref{Fu77} and \eqref{Fu80} yields
	\begin{equation}\label{Fu81}
		\E\left(|\Delta_{n,\alpha}|^4\mathbf{1}(|\Delta_{n,\alpha}|>a)\right)\le \dfrac{4^4\alpha^3}{n^{2e(p+e)}}\le \dfrac{256}{n^{2e(p+e)}}\left(\dfrac{n}{\log^2 n}\right)^3\le \dfrac{256}{n^{2ep}}
	\end{equation}
	and
	\begin{equation}\label{Fu82}
		\begin{aligned}
			\E\left(|\Delta_{n,\alpha}|^{2p}\right)&\le a^{2p}+ 4^{2p}\alpha^p {\mathbb{P}}_{\alpha}\left(|\Delta_{n,\alpha}|>a\right)\\
			&\le C(p)\left(\dfrac{1}{n^p}+\dfrac{\alpha^{p+1}}{n^{2e(p+e)}}\right)\\
			&\le \dfrac{C(p)}{n^p}.
		\end{aligned}
	\end{equation}
	
	Recall that $(W_{n,\alpha}, W_{n,\alpha}')$ is an exchangeable pair
	that satisfies \eqref{linear} with $\tau=2/n$ (see Fulman \cite[Proposition 6.1]{fulman2004stein}).
	We are now able to apply Theorem \ref{thm.main11} with $a=2e\sqrt{2}(p+e)/\sqrt{n-1}$ and $R=0$.
	Th\`{a}nh \cite{thanh2024moment} proved that if $(W,W')$ is an exchangeable pair satisfying \eqref{linear}, then
	\begin{equation*}
		\E|W|^{2p}\le \dfrac{C(p)\E|W-W'|^{2p}}{\tau^{p}}.
	\end{equation*}
	Applying this result for $(W_{n,\alpha}, W_{n,\alpha}')$, $\tau=2/n$, and using \eqref{Fu82}, we obtain
	\begin{equation}\label{moment.thanh}
		\E|W_{n,\alpha}|^{2p}\le C(p).
	\end{equation}
	From Proposition 6.5 of Fulman \cite{fulman2004stein}, we have
	\begin{equation}\label{Fu85}
		\E\left(1- \dfrac{1}{2\tau} \E(\Delta_{n,\alpha}^2|W_{n,\alpha})\right)^2\le \dfrac{3n+2\alpha}{4n(n-1)}\le \dfrac{5}{4(n-1)}.
	\end{equation}
	We also have
	\begin{equation}\label{Fu87}
		a+\dfrac{a^3}{\tau}\le \dfrac{2e\sqrt{2}(p+e)}{\sqrt{n-1}}+\left(\dfrac{2e\sqrt{2}(p+e)}{\sqrt{n-1}}\right)^3\dfrac{n}{2}\le \dfrac{C(p)}{\sqrt{n}}.
	\end{equation}
	By collecting the bounds in \eqref{Fu81}, \eqref{moment.thanh}--\eqref{Fu87}, we obtain \eqref{Fu01}.
	
	For the case where $\log^2n/n\le \alpha<1$, we have $1<1/\alpha\le n/\log^2n$. We then obtain \eqref{Fu01} by noting that (see Chen et al. \cite[p. 464]{chen2021error})
	\[\mathbb{P}_{\alpha}(W_{n,\alpha}=z)=\mathbb{P}_{1/\alpha}(W_{n,1/\alpha}=-z).\]
\end{proof}

\begin{remark}
	In view of \eqref{Fu80}, we can not apply Theorem \ref{thm.eich} (Theorem 2.3 of Eichelsbacher \cite{eichelsbacher2024stein}) to obtain a useful bound.
\end{remark}

\appendix
\section{}\label{sec.proof.lem} 


In this section, we will present a remark on Jack$_\alpha$ measures, and the proofs of Lemmas \ref{lem.bound.moment.W}--\ref{lem.bsg}.

\begin{remark}\label{rem.Jackmeasure}
	Let $\mathcal{P}_n$ be the set of all partitions of size $n$. Jack$_\alpha$ measure is a probability measure on $\mathcal{P}_n$, that is,
	\begin{equation}\label{jack.probability.measure}
		\sum_{\lambda\in \mathcal{P}_n}\mathbb{P}_{\alpha}(\lambda)=1.
	\end{equation}
	This follows from Equation (10.32) in Macdonald \cite[VI.10]{macdonald1995symmetric}.
	We present it here for completeness.
	We use the notation from Macdonald \cite{macdonald1995symmetric}.
	For $\lambda\in\mathcal{P}_n$, let $m_i(\lambda)$ be the number of rows of
	$\lambda$ of size $i$, let $l(\lambda)$ denote the total number of rows of $\lambda$, and let
	\[z_{\lambda}=\Pi_{i\ge 1}i^{m_i(\lambda)}m_i(\lambda)!,\]
	\[c_{\lambda}(\alpha)=\Pi_{s\in \lambda}(\alpha a(s)+l(s)+1),\]
	and
	\[c_{\lambda}'(\alpha)=\Pi_{s\in \lambda}(\alpha a(s)+l(s)+\alpha).\]
	Let $(1^n)$ be the partition $(1,\ldots,1)$ of $n$. Then $m_1(1^n)=n$, $m_i(1^n)=0$ for all $i\ge 2$, $l(1^n)=n$ and $z_{(1^n)}=n!$. From Equation (10.32)
	in Macdonald \cite[VI.10]{macdonald1995symmetric}, by choosing $\rho=\sigma=(1^n)$ and using (10.29) ibidem, we obtain
	\[\sum_{\lambda\in \mathcal{P}_n}\dfrac{1}{c_{\lambda}(\alpha)c_{\lambda}'(\alpha)}=\dfrac{1}{\alpha^n n!},\]
	which is equivalent to
	\[\sum_{\lambda\in \mathcal{P}_n}\dfrac{\alpha^n n!}{\Pi_{s\in \lambda}(\alpha a(s)+l(s)+1)(\alpha a(s)+l(s)+\alpha)}=1,\]
	thereby establishing \eqref{jack.probability.measure}.
\end{remark}

In the remainder of this section, we will prove Lemmas \ref{lem.bound.moment.W}--\ref{lem.bsg}.

\begin{proof}[Proof of Lemma \ref{lem.bound.moment.W}]
	By Lyapunov's inequality, we have
	\begin{equation}\label{estimate.moment.W03}
		\begin{aligned}
			\E|W+c|^q&\le \max\{1,2^{q-1}\}\left(\E|W|^q+|c|^q\right)\\
			&\le 2^{2p-1}\left((\E|W|^{2p})^{q/(2p)}+100^{2p}\right).
		\end{aligned}
	\end{equation}
	If $\E|W|^{2p}\le 1$, then \eqref{estimate.moment.W01} follows from \eqref{estimate.moment.W03} by
	choosing $C(p)= 200^{2p}$. If $\E|W|^{2p}> 1$, then $(\E|W|^{2p})^{q/(2p)}\le \E|W|^{2p}$.
	Thus, \eqref{estimate.moment.W01} also follows from \eqref{estimate.moment.W03} by
	choosing $C(p)= 200^{2p}$.
\end{proof}

\begin{proof}[Proof of Lemma \ref{lem.concentration.inequality}]
	Let \[f(x) = \begin{cases}
		0 &\text{ for } x < z-2a,  \\
		(1+x+2a)^p(x-z+2a) &\text{ for } z-2a \le x \le z+2a,  \\
		4a(1+x+2a)^p &\text{ for }x > z+2a.  \\ 
	\end{cases} \]
	Then $f$ is continuous, increasing, $f' \geq 0$ on $\mathbb{R}$, and 
	\begin{equation}\label{derivative.of.f}
		\begin{aligned}
			f'(x) \ge (1+x+2a)^p \ge (1+z)^p \text{ for }z-2a \le x \le z+2a.
		\end{aligned}
	\end{equation}
	We have from \eqref{s} that
	\begin{equation}\label{sp}
		\begin{aligned}
			\E\left((W-W')(f(W)-f(W'))\right)&= 2\tau \E Wf(W)-2\E f(W)R.
		\end{aligned}
	\end{equation}
	We estimate the left hand side of $\eqref{sp}$ as follows.
	\begin{equation}\label{LHS} 
		\begin{aligned}
			\E\left((W-W')(f(W)-f(W'))\right)
			&= \E\Delta \int_{-\Delta}^0 f'(W+t)\dx t \\
			&\ge \E\Delta \int_{-\Delta}^0 \mathbf{1}(|t|\le a)\mathbf{1}(z-a\le W \le z+a)f'(W+t)\dx t \\
			&\ge \E\Delta \int_{-\Delta}^0 \mathbf{1}(|t|\le a)\mathbf{1}(z-a \le W \le z+a)(1+z)^p\dx t \\
			&\ge (1+z)^p\E\Delta \int_{-\Delta}^0 \mathbf{1}(|\Delta|\le a)\mathbf{1}(z-a \le W \le z+a)\dx t \\
			&= (1+z)^p \E\left(\Delta^2\mathbf{1}(|\Delta|\le a)\mathbf{1}(z-a \le W \le z+a)\right),
		\end{aligned}
	\end{equation}
	where we have applied \eqref{derivative.of.f} in the second inequality.
	For the right hand side of $\eqref{sp}$, we have
	\begin{equation}\label{RHS} \begin{aligned} 
			2\tau \E Wf(W)-2\E f(W)R
			&\le 2\tau \E(4a |1+W+2a|^p|W|)+2\E(4a|1+W+2a|^p|R|)\\
			&\le C(p)a \left(\tau\sqrt{\E|W+1+2a|^{2p}\E W^2}+\sqrt{\E|W+1+2a|^{2p}}\sqrt{\E R^2}\right) \\
			&\le C(p)(1+\E|W|^{2p})a(\tau+\sqrt{\E R^2}),
		\end{aligned} 
	\end{equation}
	where we have applied the definition of $f$ in the first inequality, the Cauchy--Schwarz inequality in the second inequality,
	and Lemma \ref{lem.bound.moment.W} in the third inequality.	Combining \eqref{sp}, \eqref{LHS} and \eqref{RHS} yields \eqref{T322a}.
	
	The proof of the lemma is completed. 
\end{proof}

The proofs of Lemmas \ref{lem.bound.solutions1} and \ref{lem.bsg} are 
quite technical and similar to those of Lemmas 5.1 and 5.2 of Chen and Shao \cite{chen2001non}.
We use
the following well-known inequality (see, e.g., Chen et al. \cite[p. 16]{chen2011normal}):
\begin{equation}\label{mills.ratio}
	1-\Phi(z)\le \min\left(\dfrac{1}{2},\dfrac{1}{z\sqrt{2\pi}}\right)e^{-z^2/2},\ z>0.
\end{equation}

\begin{proof}[Proof of Lemma \ref{lem.bound.solutions1}]
	By Lemma 2.2 of Chen et al. \cite{chen2011normal}, we have 
	\begin{equation}\label{sol03}
		f_{z}(w)  =
		\begin{cases}
			\sqrt{2\pi}e^{w^2/2}(1-\Phi(w))\Phi(z),  &\text{ if } w> z,\\  
			\sqrt{2\pi}e^{w^2/2}\Phi(w)(1- \Phi(z)), &\text{ if } w \le z.\\
		\end{cases}
	\end{equation}
	By using \eqref{mills.ratio}, we obtain
	$$e^{w^2/2}(1-\Phi(w)) \le \dfrac{1}{2} \text{ for } w\ge 0.$$
	Therefore, for all $w\le 0$,
	\begin{equation}\label{sol05}
		\begin{aligned}
			f_{z}(w)& = \sqrt{2\pi}e^{w^2/2}\Phi(w)(1- \Phi(z))\\
			& =\sqrt{2\pi}e^{w^2/2}(1-\Phi(-w))(1- \Phi(z))\\
			& \le \dfrac{\sqrt{2\pi}}{2}(1-\Phi(z)).
		\end{aligned}
	\end{equation}
	By Lemma 2.3 of Chen et al. \cite{chen2011normal}, we also have 
	\begin{equation}\label{sol07}
		0 < f_{z}(w) \le \sqrt{2\pi}/4,\ w\in\R.
	\end{equation}
	To bound $\E(f_{z}(W))^2$, we write 
	\begin{equation}\label{estimate.lem1.03}
		\E (f_{z}(W))^2 =I_1+I_2+I_3,
	\end{equation}
	where
	\begin{equation*}
		\begin{aligned}
			I_1&=\E (f_{z}(W))^2\mathbf{1}(W\le 0),\
			I_2=\E (f_{z}(W))^2\mathbf{1}(W>z/2),\
			I_3=\E (f_{z}(W))^2\mathbf{1}(0<W\le z/2).
		\end{aligned}
	\end{equation*}
	By noting that $z>5$ and using \eqref{mills.ratio}, \eqref{sol05}, we have
	\begin{equation}\label{estimate.lem1.05}
		\begin{aligned}
			I_1&\le \left(\frac{\sqrt{2\pi}}{2}\right)^2(1-\Phi(z))^2 \\
			&\le \left(\dfrac{e^{-z^2/2}}{2z}\right)^2\\
			&\le \dfrac{C(p)}{(1+z)^{2p}}.
		\end{aligned} 
	\end{equation}
	By using $\E|W|^{2p}\le c_0<\infty$, \eqref{sol07}, Markov's inequality and Lemma \ref{lem.bound.moment.W}, we have
	\begin{equation}\label{estimate.lem1.07}
		\begin{aligned}
			I_2&\le \left(\frac{\sqrt{2\pi}}{4}\right)^2 \mathbb{P}(|1+W|>1+z/2)\\
			&\le \left(\frac{\sqrt{2\pi}}{4}\right)^2 \frac{\E|1+W|^{2p}}{(1+z/2)^{2p}}\\
			&\le \dfrac{C(p)(1+\E|W|^{2p})}{(1+z)^{2p}}.
		\end{aligned} 
	\end{equation}
	By noting that $z>5$ and using \eqref{mills.ratio} and \eqref{sol03}, we have
	\begin{equation}\label{estimate.lem1.08}
		\begin{aligned}
			I_3&\le 2\pi(1-\Phi(z))^2\E \left(e^{W^2}\mathbf{1}(0 < W \le z/2)\right)\\
			&\le \dfrac{\pi e^{-z^2}e^{z^2/4}}{2}\\
			&=\dfrac{\pi e^{-3z^2/4}}{2}\le \dfrac{C(p)}{(1+z)^{2p}}.
		\end{aligned} 
	\end{equation}
	Combining \eqref{estimate.lem1.03}--\eqref{estimate.lem1.08} yields \eqref{estimate.lem1.01}.
	
	The proof of \eqref{estimate.lem1.02} can proceed similarly. Since $0\le |\xi| \le |\Delta|$,
	\begin{equation}\label{sol12}
		\E|\xi|^{2p} \le \E|W-W'|^{2p} \le 2^{2p-1} (\E|W|^{2p}+\E|W'|^{2p})=2^{2p}\E|W|^{2p}.
	\end{equation}
	From Stein \cite[p. 23--24]{stein1986approximate}, we have following facts:
	\begin{equation}\label{sol13}
		0<f_{z}'(w)  < (1-\Phi(z))\text{ for }w\le 0,
	\end{equation}
	\begin{equation}\label{sol15}
		f_{z}'(w) = (1-\Phi(z))(1+\sqrt{2\pi}w e^{w^2/2}\Phi(w)) \text{ for } w\le z,
	\end{equation}
	and
	\begin{equation}\label{sol17}
		|f_{z}'(w)| \le 1 \text{ for all } w.
	\end{equation}
	To bound $\E(f_{z}'(W+\xi))^2$, we write 
	\begin{equation}\label{estimate.lem1.13}
		\E(f_{z}'(W+\xi))^2 =J_1+J_2+J_3,
	\end{equation}
	where
	\begin{equation*}
		\begin{aligned}
			J_1&=\E (f_{z}'(W+\xi))^2\mathbf{1}(W+\xi\le 0),\\
			J_2&=\E (f_{z}'(W+\xi))^2\mathbf{1}(W+\xi>z/2),\\
			J_3&=\E (f_{z}'(W+\xi))^2\mathbf{1}(0<W+\xi\le z/2).
		\end{aligned}
	\end{equation*}
	By noting that $z>5$ and using \eqref{mills.ratio} and \eqref{sol13}, we have
	\begin{equation}\label{estimate.lem1.15}
		\begin{aligned}
			J_1&\le (1-\Phi(z))^2 \le \left(\dfrac{e^{-z^2/2}}{2}\right)^2\le \dfrac{C(p)}{(1+z)^{2p}}.
		\end{aligned} 
	\end{equation}
	By using $\E|W|^{2p}<\infty$, \eqref{sol12}, \eqref{sol17} and Markov's inequality, we have
	\begin{equation}\label{estimate.lem1.17}
		\begin{aligned}
			J_2&\le \mathbb{P}(|1+W+\xi|>1+z/2)\\
			&\le \frac{\E|1+W+\xi|^{2p}}{(1+z/2)^{2p}}\\
			&\le 3^{2p-1}\dfrac{(\E|W|^{2p}+\E|\xi|^{2p}+1)}{(1+z/2)^{2p}} \\
			&\le \dfrac{C(p)(1+\E|W|^{2p})}{(1+z)^{2p}}.
		\end{aligned} 
	\end{equation}
	By noting that $z>5$ and using \eqref{mills.ratio} and \eqref{sol15}, we have
	\begin{equation}\label{estimate.lem1.18}
		\begin{aligned}
			J_3&\le (1-\Phi(z))^2\E \left((1+\sqrt{2\pi}(W+\xi)e^{(W+\xi)^2/2})^2\mathbf{1}(0 < W+\xi \le z/2)\right)\\
			&\le \dfrac{ e^{-z^2}(1+\sqrt{2\pi}ze^{z^2/8})^2}{4}\\
			&\le \dfrac{C(p)}{(1+z)^{2p}}.
		\end{aligned} 
	\end{equation}
	Combining \eqref{estimate.lem1.13}--\eqref{estimate.lem1.18} yields \eqref{estimate.lem1.02}.
	The proof of the lemma is completed.
\end{proof}

\begin{proof}[Proof of Lemma \ref{lem.bsg}]
	The proof of Lemma \ref{lem.bsg} is similar to those of \eqref{estimate.lem1.01} and \eqref{estimate.lem1.02}.
	From definition of $g_{z}$, we have (see Chen and Shao \cite{chen2001non}, p. 248)
	\begin{equation}\label{estimate.lem1.24}
		\begin{aligned}
			g_{z}(w)  =
			\begin{cases}
				\left( \sqrt{2\pi}(1+w^2)e^{w^2/2}(1-\Phi(w)) - w \right)\Phi(z) &\text{ if } w\ge z,\\  
				\left( \sqrt{2\pi}(1+w^2)e^{w^2/2}\Phi(w)+w\right) (1- \Phi(z)) &\text{ if } w < z.\\
			\end{cases}
		\end{aligned} 
	\end{equation}
	Chen and Shao \cite[p.249]{chen2001non} proved that 
	\begin{equation}\label{estimate.lem1.25}
		\begin{aligned}
			g_{z} \ge 0,\ g_{z}(w) \le 2(1-\Phi(z)) \text{ for $w\le 0$ and $g$ is increasing for $0 < w \le z$}.
		\end{aligned} 
	\end{equation}
	By \eqref{sol07}, \eqref{sol17} and Lemma \ref{lem.bound.moment.W}, we have
	\begin{equation}\label{estimate.lem1.21}
		\begin{aligned}
			\E g_{z}^2(W+u) &\le 2\E\left(f_{z}^2(W+u) + (W+u)^2(f_{z}'(W+u))^2\right)\\
			&\le {\pi}/4 + 2\E(W+u)^2 \\
			&\le C(p)(1+\E|W|^{2p}).
		\end{aligned} 
	\end{equation}
	Applying \eqref{estimate.lem1.21}, the Cauchy--Schwarz inequality, Markov's inequality and Lemma \ref{lem.bound.moment.W}, we obtain
	\begin{equation}\label{estimate.lem1.23}
		\begin{aligned}
			\E g_{z}(W+u)\mathbf{1}(W+u > z/2)
			&\le (\E g_{z}^2(W+u))^{1/2}(\E\mathbf{1}(W+u > z/2))^{1/2}\\
			&\le \dfrac{(\E g_{z}^2(W+u))^{1/2}\left(\E|W+2|^{2p}\right)^{1/2}}{(1+z/2)^p}\\
			&\le \dfrac{C(p)(1+\E|W|^{2p})}{(1+z)^p}.
		\end{aligned} 
	\end{equation}
	It thus follows that
	\begin{align*}
		\E g_{z}(W+u)&=\E g_{z}(W+u)\left(\mathbf{1}(W+u \le 0)+\mathbf{1}(0 < W+u \le z/2)+\mathbf{1}(W+u > z/2)\right)\\
		&\le 2(1-\Phi(z)) + g_{z}(z/2) + \E(g_{z}(W+u)\mathbf{1}(W+u> z/2))\\
		&\le 2(1-\Phi(z)) + \left(\sqrt{2\pi}(1+z^2/4)e^{z^2/8}\Phi(z/2)+z/2\right)(1-\Phi(z))+\dfrac{C(p)(1+\E|W|^{2p})}{(1+z)^p}\\
		&\le e^{-z^2/2} + C\left((1+z^2)e^{z^2/8}+z\right)e^{-z^2/2} + \dfrac{C(p)(1+\E|W|^{2p})}{(1+z)^p}\\
		&\le \dfrac{C(p)(1+\E|W|^{2p})}{(1+z)^p},
	\end{align*}
	where we have applied \eqref{estimate.lem1.25} in the first inequality, \eqref{estimate.lem1.24}
	and \eqref{estimate.lem1.23} in the second inequality, and \eqref{mills.ratio} in the third inequality. This ends the proof of the lemma.
\end{proof}

\vskip.2in	
\textbf{Funding:} This work was supported by the National Foundation for Science and Technology Development (NAFOSTED), grant no. 101.03-2021.32.
\small	

\providecommand{\bysame}{\leavevmode\hbox to3em{\hrulefill}\thinspace}
\providecommand{\MR}{\relax\ifhmode\unskip\space\fi MR }
\providecommand{\MRhref}[2]{%
	\href{http://www.ams.org/mathscinet-getitem?mr=#1}{#2}
}
\providecommand{\href}[2]{#2}

\end{document}